\documentclass[11pt]{amsart}

\usepackage[psamsfonts]{eucal}
\usepackage{amsmath, amssymb, amsfonts}
\usepackage{setspace}
\usepackage{enumerate}
\usepackage[bookmarks, colorlinks=true, pdfstartview=FitV, linkcolor=blue, citecolor=green, urlcolor=blue]{hyperref}
\usepackage{cite}  

\newcommand{\noi}{\noindent}

\newtheorem{thm}{Theorem}[section]
\newtheorem{cor}[thm]{Corollary}
\newtheorem{proposition}[thm]{Proposition}
\newtheorem{prop}[thm]{Proposition}

\newtheorem{lemma}[thm]{Lemma}

\theoremstyle{definition}

\newtheorem{remark}[thm]{Remark}

\begin{document}

\title{Families of spectral sets for Bernoulli convolutions}

\author{Palle E. T. Jorgensen, Keri A. Kornelson, and Karen L. Shuman}

\address{\textrm{(P. E. T. Jorgensen)}
Department of Mathematics,
University of Iowa,
Iowa City, IA 52242 USA}
\email{jorgen@math.uiowa.edu}

\address{\textrm{(K. A. Kornelson)}
Department of Mathematics,
University of Oklahoma,
Norman, OK 73019 USA}
\email{kkornelson@math.ou.edu}

\address{\textrm{(K. L. Shuman)}
Department of Mathematics \& Statistics,
Grinnell College,
Grinnell, Iowa 50112 USA}
\email{shumank@math.grinnell.edu}

\thanks{The research of the first author was supported in part by grants from the National Science Foundation.  The research of the second and third authors was supported in part by NSF grant DMS0701164. }

\date{\today}
\subjclass[2000]{Primary 28A80, 42A16; Secondary 42C25, 46E30, 42B05, 28D05}

\keywords{Bernoulli convolution, spectral measure, Hilbert space, fractals, Fourier series, Fourier
coefficients, orthogonal series, iterated function system}
\begin{abstract}
We study the harmonic analysis of Bernoulli measures $\mu_\lambda$, a one-parameter family of
compactly supported Borel probability measures on the real line. The parameter $\lambda$ is a fixed
number in the open interval $(0, 1)$. The measures $\mu_\lambda$ may be understood in any one of
the following three equivalent ways: as infinite convolution measures of a two-point probability
distribution; as the distribution of a random power series; or as an iterated function system (IFS) equilibrium measure
determined by the two transformations $\lambda(x \pm 1)$.  For a given $\lambda$, we consider
the harmonic analysis in the sense of Fourier series in the Hilbert space $L^2(\mu_\lambda)$.  For
$L^2(\mu_\lambda)$ to have infinite families of orthogonal complex exponential functions $e^{2\pi i s (\cdot)}$,
 it is known that $\lambda$ must be a rational number of the form $\frac{m}{2n}$, where $m$ is odd. We show that $L^2(\mu_{\frac{1}{2n}})$ has a variety of Fourier bases; i.e. orthonormal bases of exponential functions.  For some other rational values of $\lambda$, we exhibit maximal Fourier families that are not orthonormal bases. 
\end{abstract}
\maketitle

\section{Introduction}\label{Sec:Introduction}

Fractal scaling and self-similarity occur both in nature and in 
man-made systems such as large communication networks.  The applications of this theory range from the study of biological systems to the development of models for telecommunications.  
To better understand this fractal
scaling we examine a class of fractals constructed by a
 finite family of affine transformations.   The iteration of these transformations produce compact fractals embedded in the real line and fractal measures $\mu$ supported on these sets.  Another type of  iteration
produces spectra, which correspond to orthogonal families in
$L^2(\mu)$.
    
    The Bernoulli convolutions are a class of these fractal measures.  They consist of a one-parameter family of 
compactly supported Borel probability measures $\mu_\lambda$, where the parameter $\lambda$ is taken from the interval $(0, 1)$.  Of all the affine fractal measures, Bernoulli convolutions
 have been explored the most.
Despite the rich literature on Bernoulli convolutions, the
corresponding Fourier analysis of even this restricted family is still in its
infancy. 

Historically three types of questions have often been considered: For what values of $\lambda$  is $\mu_\lambda$
absolutely continuous?  For what values of $\lambda$ does $L^2(\mu_{\lambda})$ admit a Fourier-based harmonic
analysis? When it does, what are the Fourier bases in the Hilbert space $L^2(\mu_{\lambda})$?  In this paper we
are concerned with the last two questions.  We examine infinite and recursively generated orthogonal
systems of Fourier exponentials in the Hilbert spaces $L^2(\mu)$ when  $\mu$ is one
of the Bernoulli measures and examine when such systems are orthonormal bases.

\subsection{Background and definitions}\label{Subsec:Background}
Given $S$ a subset of $\mathbb{R}$, we will use the notation $E(S)$ to be the set of exponential functions 
\begin{equation}\label{Eqn:E(S)}
E(S) := \{ e_s(x) := \exp (2 \pi i s x) \:|\: s\in S\}.
\end{equation} 
We use terminology here motivated by classical Fourier duality. 
We say that $S$ is \textit{orthogonal} if the exponential functions in $E(S)$ are orthogonal in $L^2(\mu)$. We say that $S$ yields a \textit{Fourier frame} if $E(S)$ satisfies a frame estimate.  When $E(S)$ forms an orthonormal basis (ONB) in $L^2(\mu)$, we say that the pair $(\mu, S)$ is a \textit{spectral pair}, and we say that $S$ is a \textit{spectral set} in $L^2(\mu)$.

The measures we discuss in this paper come from the one-parameter family of Bernoulli measures $\{\mu_{\lambda} : \lambda \in (0,1)\}$.  In the literature, there are three different ways in which to produce a Bernoulli measure, each of which we mention briefly here.  Let $\lambda \in (0,1)$.  
\renewcommand{\labelenumi}{(\roman{enumi})}
\begin{enumerate}
\item  \cite{Erd39, Erd40}  Let $\Omega = \prod_{1}^{\infty} (\pm 1)$ be the product space with the Bernoulli coin-tossing measure $\mathbb{P}$, and define the random variable \[\mathbb{X}_{\lambda}(\omega) = \sum_{k=1}^{\infty} \omega_k \lambda^k, \; \textrm{for}\;  \omega = (\omega_k) \in \Omega.\]
 Then for a Borel subset $E \subset \mathbb{R}$, we define \[ \mu_{\lambda}(E) := \mathbb{P}(\mathbb{X}_{\lambda}^{-1}(E)); \] i.e. $\mu_{\lambda}$ is the distribution of the random variable $\mathbb{X}_{\lambda}$. 
 \item  \cite{Hut81}  Define the two real-valued maps  \[ \tau_{+}(x) = \lambda (x + 1) \; \textrm{and}\; \tau_-(x) = \lambda(x-1).\]  These form an affine iterated function system (IFS).  The measure $\mu_{\lambda}$ is the unique equilibrium measure satisfying the invariance property 
 \begin{equation}\label{eqn:mu} \mu_{\lambda} = \frac12 ( \mu_{\lambda} \circ \tau_+^{-1} + \mu_{\lambda} \circ \tau_-^{-1}).\end{equation}
 \item  \cite{Erd39, Erd40}  The measure $\mu_{\lambda}$ has Fourier transform given by an infinite product formula which we present in Equation (\ref{eqn:muhat}).\end{enumerate}
The support of the Bernoulli measure $\mu_{\lambda}$ is the unique compact set $K_{\lambda} \subset \mathbb{R}$ which satisfies $K_{\lambda} = \tau_+(K_{\lambda}) \cup \tau_-(K_{\lambda})$.  In the random variable construction,  $K_{\lambda} = \mathbb{X}_{\lambda}(\Omega)$.  For every $\lambda \in (0,1)$, the set $K_{\lambda}$ is contained in the interval $\Bigl[ -\frac{\lambda}{1-\lambda}, \frac{\lambda}{1-\lambda}\Bigl]$.

In the cases where $\lambda$ is the reciprocal of an integer, the corresponding measure $\mu_{\lambda}$ is one of the Cantor measures, and its support is a Cantor set.   For example, if $\lambda = \frac{1}{3}$, then we have the familiar middle-third Cantor construction in which an initial interval is subdivided in $3$ equal parts, and the middle third is omitted. The procedure is continued in a recursive algorithm involving a step-by-step rescaling of the measure. The limiting measure is our $\mu_{\frac{1}{3}}$.  

When $\lambda = \frac{1}{2}$, the measure $\mu_{\frac{1}{2}}$ is Lebesgue measure on the interval $[-1,1]$, scaled so that the total measure is $1$.   In this case, we can take the set $S$ to be  $\frac{1}{2}\mathbb{Z}$ to produce the orthonormal basis  $E(S)$ for $L^2(\mu_{\frac12})$, so we have a spectral pair $(\mu_{\frac12}, \frac12\mathbb{Z})$.  At the other extreme, when $\lambda = \frac{1}{3}$,  it turns out (see \cite{JoPe98}) that $L^2(\mu_{\frac{1}{3}})$ does not even contain an orthogonal set $E(S)$  with cardinality more than $2$.  

     It was proved in \cite{JoPe98} that even though $L^2(\mu_{\frac{1}{3}})$ does not have a Fourier basis,  $L^2(\mu_{\frac{1}{4}}$) does.     In fact, the general result in \cite{JoPe98} is that $L^2(\mu_{\frac{1}{2n}})$ has a Fourier basis $E(S)$ for all $n\in\mathbb{N}$.      It has been shown in \cite{JKS08, HuLa08} that there can be infinite orthogonal collections of exponentials when $\lambda > \frac12$, but a recent result in \cite{DHJ09} has shown that  Fourier bases are impossible for these cases.  In this paper, we take values of $\lambda$ which are known to have orthonormal bases and explore the families of Fourier bases which can be formed for $L^2(\mu_{\lambda})$.

\subsection{Main results}
We will use a spectral function to determine the Fourier basis properties for collections of exponential functions.  Theorem \ref{Thm:cGammaAnalytic} contains the most general setting (frames) for which the spectral function has an entire analytic extension.  The lemmas leading up to this result contain detailed results of independent interest.

In Section \ref{sec:spectra}, we demonstrate that the spectral function is a Perron-Frobenius eigenvector for an associated transfer operator $T$.   The main results entail sufficient conditions under which the transfer operator is contractive and hence corresponds to a spectral pair for a Bernoulli measure $\mu_{\frac{1}{2n}}$.  First, we consider the special case $\lambda = \frac18$.    It was shown in \cite{JoPe98} that given the set  
\begin{equation}\label{eqn:Gamma} \Gamma\Bigl(\frac{1}{8}\Bigr)
=\Biggl\{ 
\sum_{k=0}^{\textrm{finite}} a_k 8^k : a_k \in \{0, 2\} 
\Biggr\} 
=  \{0, 2, 16, 18, \ldots ,\},
\end{equation} 
the collection $E(\Gamma(\frac18))$ forms an ONB for $L^2(\mu_{\frac{1}{8}})$.

\noindent\textsc{Theorem A.  } The set  $E(3\Gamma(\frac18))$ forms an ONB for $L^2(\mu_{\frac{1}{8}})$. 

We then generalize the result to the case of $\lambda = \frac{1}{2n}$ and $p$ an odd positive integer.  The set $E(\Gamma(\frac{1}{2n}))$ where  

\begin{equation}\label{Eqn:Gamma12n} \Gamma\Bigl(\frac{1}{2n}\Bigr)
=\Biggl\{ 
\sum_{k=0}^{\textrm{finite}} a_k (2n)^k : a_k \in \Bigr\{0, \frac{n}{2}\Bigr\} 
\Biggr\} \end{equation} is known by \cite{JoPe98} to be an ONB for $L^2(\mu_{\frac{1}{2n}})$.

\noindent\textsc{Theorem B.  }  If $p \in 2\mathbb{N}+1$ such that  $p<\frac{2(2n-1)}{\pi}$, then $\left(\mu_{\frac{1}{2n}}, p\Gamma\Big(\frac{1}{2n}\Big)\right)$ is a spectral pair for $L^2(\mu_{\frac{1}{2n}})$.
  
In Section \ref{Sec:MaximalNotONB}, we examine cases when $L^2(\mu_{\lambda})$ may not have an ONB of exponential functions.  Theorem \ref{Thm:38Maximality} demonstrates cases where there do exist maximally orthogonal families.

\subsection{Overview of prior literature}\label{subsec:refs}  

The results in this paper relate to general Fourier duality results in analysis.  This general framework includes subjects such as universal tiling sets and the dual spectral set conjecture (see \cite{DHS09, Jor06, JKS07a, JoPe99, PeWa01}).  Starting with \cite{Fug74}, questions about Fourier duality have received considerable attention with respect to pure harmonic analysis \cite{DuJo07b, DHPS08, DuJo09, JoPe98, Ped04a, Ped04b, LaWa02, LaWa06} and with respect to applications such as wavelets, sampling, algorithms, martingales, and substitution-dynamical systems \cite{DuJo05, DuJo06, DuJo07a, JoSo07}.

     The special subclass of affine IFSs and their harmonic analysis was initiated by one of the authors with Pedersen \cite{JoPe98}.  These special affine IFSs have been extensively researched since---see, for example, \cite{PeSo96, PSS00, Sid03, Den09, HuLa08, DuJo06, DuJo07a, DuJo07b, DuJo07c,  DuJo09, Jor06, JKS07a, JKS08,  Str98, MaSi98, HuLa02}.  While these papers place the focus on the possibilities for orthogonal sets of Fourier frequencies, there are more recent analysis of overcomplete bases, more particularly, systems of Fourier frames in $L^2(\mu)$ in the case $\mu$ is one of the IFS measures. In this case, the results so far are only at an initial stage. For example, if $\mu$ is the middle-third Cantor measure, the question of whether of not $L^2(\mu)$ admits Fourier frames has implications for the Kadison-Singer conjecture. For some of these results regarding frames, the reader is referred to \cite{OrSe02, FJKO05, HKLW07, Jor08, JoOk08}.

\section{The function $c_{\Gamma}$ and some elementary properties}\label{Sec:BasesAndFrames}

Given a Bernoulli measure $\mu_\lambda$, the Hilbert space $L^2(\mu_\lambda)$ will only have an orthogonal Fourier basis $E(\Gamma)$ for
certain values of the parameter $\lambda$.   As previously mentioned, when $\lambda= \frac13$  there
there cannot be more than two orthogonal exponential functions,  and if $\lambda > \frac12$  then the $\lambda$-IFS has
overlap and the measure cannot be spectral \cite{DHJ09}.  If  $\lambda = \frac{1}{2n}$ however, then $L^2(\mu_{\frac{1}{2n}})$ does have orthogonal Fourier bases.   Specifically, given the set $\Gamma(\frac{1}{2n})$ defined by \begin{equation}\label{eqn:ngamma} \Gamma\Bigl(\frac{1}{2n}\Bigr)
:=\Biggl\{ 
\sum_{k=0}^{\textrm{finite}} a_k (2n)^k : a_k \in \Bigl\{0, \frac{n}{2}\Bigr\} 
\Biggr\}, \end{equation}  
it is proved in \cite{JoPe98} that $E(\Gamma)$ is an ONB for  $L^2(\mu_{\frac{1}{2n}})$.

 We will show in Section \ref{sec:spectra} that for a fixed $n$, there can be a variety of such bases.  We establish a condition such that for any odd integer $p$ satisfying the condition, the $p$-dilated set $p\Gamma$ is a spectrum for  $L^2(\mu_{\frac{1}{2n}})$.  
 
 In this section, we describe a real-valued function $c$ determined by the Bernoulli measure and the set $\Gamma$.  This function is a tool for determining the properties of the exponential set $E(\Gamma)$ in the Hilbert space $L^2(\mu_{\lambda})$.  We go on to show that this function has an entire analytic extension to the complex plane and also that it has a bounded derivative.  These are crucial properties used in the proof of our main result in Section \ref{sec:spectra}.

Let $\mu_{\lambda}$ be a Bernoulli convolution measure on $\mathbb{R}$.  We can define the Fourier transform of the measure, i.e. the function $\widehat{\mu}_{\lambda}$, by \begin{equation*} \widehat{\mu}_{\lambda}(t) = \int_{\mathbb{R}} e^{2\pi i x t} \mathrm{d}\mu_{\lambda}(x). \end{equation*}  It is readily verified that the invariance property from Equation (\ref{eqn:mu})  gives \begin{equation}\label{eqn:lambdamuhat} \widehat{\mu}_{\lambda}(t) = \cos(2\pi \lambda t) \widehat{\mu}_{\lambda}(\lambda t).\end{equation}  We can iterate this to achieve the convergent infinite product representation of this function:
 \begin{equation}\label{eqn:muhat}  \widehat{\mu}_{\lambda}(t) =  \prod_{k=1}^{\infty} \cos(2\pi \lambda^kt). \end{equation} 

The zeros of the function $\widehat{\mu}_{\lambda}$ will be useful in determining orthogonality.  We denote by $\mathcal{Z}_{\lambda}$ the set of all $t \in \mathbb{R}$ such that one of the factors in $\widehat{\mu}_{\lambda}$ is zero, which gives \begin{equation}\label{eqn:zero} \mathcal{Z}_{\lambda} = \Biggr\{ \frac{2m+1}{4 \lambda^k} \,:\, m \in \mathbb{Z}, k \geq 1\Biggr\}. \end{equation}  Observe that $\langle e_{\gamma_1},e_{\gamma_2} \rangle = \widehat{\mu}_{\lambda}(\gamma_1-\gamma_2)$.  Therefore, we have the following lemma.

\begin{lemma}\label{lem:zero} Exponential functions $e_{\gamma_1}$ and $e_{\gamma_2}$ are orthogonal in $L^2(\mu_{\lambda})$ if and only if $\gamma_1- \gamma_2 \in \mathcal{Z}_{\lambda}$.
\end{lemma}

Given a countable set $\Gamma$, we will be determining properties of the family of exponentials $E(\Gamma)$ via the \textit{spectral function} $c_{\Gamma}:\mathbb{R}\rightarrow\mathbb{R}$ defined by
\begin{equation}\label{eqn:cdefn}
c_{\Gamma}(t): = \sum_{\gamma \in \Gamma} |\langle e_{\gamma},e_{-t} \rangle|^2 =  \sum_{\gamma\in\Gamma} |\widehat{\mu}(\gamma + t)|^2.
\end{equation} The following are well-known properties of the function $c_{\Gamma}$.
\begin{lemma}\label{lem:cproperties}
\begin{enumerate}[(1)]
\item $E(\Gamma)$ is orthogonal if and only if $c_{\Gamma}({t}) \leq 1$ for all ${t}\in\mathbb{R}$.
\item $E(\Gamma)$ is an ONB if and only if $c_{\Gamma}$ is identically $1$.
\item If the set $E(\Gamma)$ has an upper frame bound $B$, then $c_{\Gamma}({t}) \leq B$ for all  ${t}\in\mathbb{R}$.
\item If the set $E(\Gamma)$ has a lower frame bound $A$, then $c_{\Gamma}({t}) \geq A$ for all  ${t}\in\mathbb{R}$.
\end{enumerate}
\end{lemma}

Denote by $\delta_{\gamma}$ the element of $\ell^2(\Gamma)$ which has $1$ in the $\gamma$ coordinate and $0$ in all other coordinates.  Then $\{\delta_{\gamma} : \gamma\in\Gamma\}$ is the standard ONB for $\ell^2(\Gamma)$. 

Let $S$ be the operator from $\ell^2(\Gamma)$ to $L^2(\mu)$
defined by
\begin{equation}\label{Eqn:Sdeltagamma}
S(\delta_{\gamma}) =  e_{\gamma} \text{ for all }\gamma\in\Gamma.
\end{equation}
In frame theory, $S$ is the \textit{synthesis operator} for $E(\Gamma)$.  If $S$ is a bounded operator, then its adjoint $S^*: L^2(\mu) \rightarrow \ell^2(\Gamma)$ is also bounded and is given by
\begin{equation*}
(S^*f)_{\gamma} = \langle e_{\gamma}, f\rangle_{L^2(\mu)}.
\end{equation*}
$S^*$ is called the \textit{analysis operator} for $E(\Gamma)$.

We can now connect $S$ to the upper frame bound of $E(\Gamma)$.

\begin{lemma}\label{Lemma:BoundedImpliesUpperFrame}
Given a measure $\mu$ with compact support and a countable set $\Gamma$, the operator $S:\ell^2(\Gamma) \rightarrow L^2(\mu)$ from Equation (\ref{Eqn:Sdeltagamma}) is bounded  if and only if the set $E(\Gamma)$ has an upper frame bound, i.e. there exists a constant $B$ such that \[\sum_{\gamma \in \Gamma} |\langle e_{\gamma}, f \rangle_{L^2(\mu)}|^2 \leq B\|f\|_{L^2(\mu)}^2 \qquad \forall f \in L^2(\mu).\]  In that case, the spectral function $c_{\Gamma}$ is bounded.
\end{lemma}

\begin{proof}
If $S$ is bounded, we exhibit the upper frame bound $B = \|S^*\|_{op}^2 = \|S\|_{op}^2$:
\begin{equation*}\label{Eqn:SStarf}
\sum_{\gamma\in\Gamma} |\langle e_{\gamma}, f\rangle_{L^2(\mu)}|^2 = \|S^* f\|_{\ell^2(\Gamma)}^2 \leq \|S^*\|_{op}^2\|f\|_{L^2(\mu)}^2 = \|S\|_{op}^2\|f\|_{L^2(\mu)}^2.
\end{equation*}
Conversely, suppose the set $E(\Gamma)$ has an upper frame bound $B$.  Define the operator $A:L^2(\mu) \rightarrow \ell^2(\Gamma)$ by
\begin{equation*}
(Af)_{\gamma} =  \langle e_{\gamma}, f\rangle_{L^2(\mu)}.
\end{equation*}
Then $A$ is bounded because
\begin{equation*}
\|Af\|_{\ell^2(\Gamma)}^2 = \sum_{\gamma\in\Gamma} |\langle e_{\gamma}, f\rangle_{L^2(\mu)}|^2 \leq B\|f\|_{L^2(\mu)}^2.
\end{equation*}
Then $A$ must have a bounded adjoint $A^*$, and $A^*$ is the synthesis operator we called $S$ above.  

The boundedness of $c_{\Gamma}$ follows from these properties and Lemma \ref{lem:cproperties}.
\end{proof}

Throughout the remainder of this section, we assume for $\mu$ and $\Gamma$ that the operator $S$ is bounded.

The operators $S$ and $S^*$ relate to the function $c_{\Gamma}$ in a variety of ways.  Let $t\in\mathbb{R}$.  Then
\begin{equation}\label{Eqn:cGammaAndA}
\begin{split}
c_{\Gamma}(t)
& = \sum_{\gamma\in\Gamma} |\widehat{\mu}(t+\gamma)|^2 
= \sum_{\gamma\in\Gamma} | \langle e_{\gamma}, e_{-t} \rangle_{L^2(\mu)}|^2\\
& = \sum_{\gamma\in\Gamma} |\langle S\delta_\gamma, e_{-t}\rangle_{L^2(\mu)}|^2
=  \sum_{\gamma\in\Gamma} |\langle \delta_\gamma, S^* e_{-t}\rangle_{\ell^2(\Gamma)}|^2\\
& \underbrace{=}_{\text{Parseval}} \|S^*e_{-t}\|^2 
= \langle e_{-t}, SS^*e_{-t}\rangle_{L^2(\mu)}.
\end{split}
\end{equation}
Give the operator $SS^*$ the name $P$; this is the \textit{frame operator} for $E(\Gamma)$. $P$ is a bounded operator because $S$ and $S^*$ are both bounded operators, and $\|P\| = \|S\|^2$.  Therefore, Equation (\ref{Eqn:cGammaAndA}) can be written
\begin{equation}\label{Eqn:cGammaP}
c_{\Gamma}(t) = \langle e_{-t}, Pe_{-t}\rangle_{L^2(\mu)}.
\end{equation}
We note that we do not know anything about the spectrum of $P$ because we do not assume the existence of a lower frame bound.

\subsection{Complex analytic properties of $c_{\Gamma}$}\label{Subsec:ComplexAnalytic}

Given a  measure $\mu$ with compact support and a countable set of real numbers $\Gamma$, the orthogonality and other basis properties of $E(\Gamma)$ with respect to $\mu$ are determined by the spectral function $c_\Gamma$ as defined in Equation (\ref{eqn:cdefn}).   In this section, we prove that the function $c_\Gamma$ has an
entire analytic extension to the whole complex plane, and we derive properties of the
extension which will be needed later.  This is not a new result, but it is perhaps not clearly portrayed in the literature.  We include an overview of the proof here for the interested reader.   

\begin{thm}\label{Thm:cGammaAnalytic}
Assume there exists a bounded synthesis operator $S:\ell^2(\Gamma) \rightarrow L^2(\mu)$ with $S(\delta_{\gamma}) = e_{\gamma}$ for all $\gamma\in \Gamma$. Then the function $c_{\Gamma}(t)$ has an entire analytic extension to the complex plane $\mathbb{C}$.
\end{thm}
\begin{proof}
We rely on a series of lemmas for the proof.

\begin{lemma}\label{lem:cphi} Let $M_x$ denote the multiplication operator $f(x) \mapsto 2\pi x f(x)$, and let $P = SS^*$ as above.  Then
\[ c_{\Gamma}(t) = \langle 1, e^{itM_x}Pe^{-itM_x}1\rangle_{L^2(\mu)}.\]
\end{lemma}
\begin{proof}
Since $\mu$ has compact support, $M_x$ is bounded for all $x\in\mathbb{R}$.  We can also define the operator $e^{itM_x}$, which maps $f(x)$ to  $e^{2\pi itx}f(x)$.  To see this, use the power series expansion of $e^{itM_x}$ to find

\begin{equation*}
[e^{itM_x}f](x) = \sum_{n=0}^{\infty} \frac{(it)^nM_x^n}{n!}f(x) = e^{2\pi itx}f(x).
\end{equation*} 

When we apply the operator $e^{itM_x}$ to the constant function $1$, we just get $e^{2\pi i tx}$.  The adjoint of the operator $e^{itM_x}$ is $e^{-itM_x}$.  Therefore we have
\begin{equation*}
c_{\Gamma}(t) 
\:\:
\underbrace{=}_{(\ref{Eqn:cGammaP})} 
\:\:
\langle e_{-t}, Pe_{-t}\rangle_{L^2(\mu)} 
= \langle 1, e^{itM_x}Pe^{-itM_x}1 \rangle_{L^2(\mu)}.
\end{equation*} \end{proof}

\begin{lemma}\label{Lemma:PowerSeries}
The operators \[e^{itM_x}Pe^{-itM_x} \quad \textrm{ and } \quad e^{it\,\rm ad\it M_x}(P)\] \emph{(}where the ad notation $[\mathrm{ad}\, A](B) $ means the commutator $[A,B]$\emph{)}
have exactly the same derivatives at $t=0$.
\end{lemma}
\begin{proof}
We will show that the two operator functions in the lemma are
analytic.  In the right-hand side above, we have conjugated the operator $P$ by the unitary operator $e^{itM_x}$.  We will now find the power series in $t$ associated with 
$e^{itM_x}Pe^{-itM_x}$.  
To write the power series, we need the derivatives
\begin{equation*}
\frac{d^n}{dt^n}\Bigl( e^{itM_x}Pe^{-itM_x} \Bigr)\Big|_{t=0}.
\end{equation*}

We have
\begin{equation}\label{Eqn:FirstDerivative}
\frac{d}{dt}\Bigl(e^{itM_x}Pe^{-itM_x}\Bigr) = e^{itM_x}(iM_x)Pe^{-itM_x} + e^{itM_x}P (-i)M_xe^{-itM_x}
\end{equation}
because the operators $M_x$ and $e^{itM_x}$ commute.
Factoring $e^{itM_x}$ and $e^{-itM_x}$ leaves the commutator $iM_xP - PiM_x = [iM_x, P]$, which can also be written as $[\text{ad}(iM_x)](P)$.

Therefore, at $t = 0$, we are left simply with the commutator
\begin{equation*}
\frac{d}{dt}\Bigl(e^{itM_x}Pe^{-itM_x}\Bigr)\Big|_{t=0} = e^{itM_x}[iM_x, P]e^{-itM_x}\Big|_{t=0}
=[\text{ad}(iM_x)](P).
\end{equation*}  

When we calculate the second derivative, the computations are exactly the same, except now $[iM_x, P]$ appears in the place of $P$ in Equation (\ref{Eqn:FirstDerivative}).  We have
\begin{equation*}
\frac{d^2}{dt^2}\Bigl( e^{itM_x}Pe^{-itM_x}\Bigr)= e^{itM_x}[iM_x, [iM_x, P]]e^{-itM_x},
\end{equation*}
and at $t=0$, the derivative is $[iM_x, [iM_x, P]] = [\text{ad}(iM_x)]^2(P)$.

In the same manner, the $n^{\textrm{th}}$ derivative at $t = 0$ is
\begin{equation*}
\frac{d^n}{dt^n}\Bigl( e^{itM_x}Pe^{-itM_x} \Bigr)\Big|_{t=0}
= [\text{ad}(iM_x)]^n(P).
\end{equation*}

Equipped with all the derivatives at $t=0$, we can now write down the power series associated with $e^{itM_x}Pe^{-itM_x}$:
\begin{equation}\label{Eqn:PowerSeriesExpP}
\sum_{n=0}^{\infty} \frac{t^n}{n!} 
\Biggl(\frac{d^n}{dt^n}\Bigl( e^{itM_x}Pe^{-itM_x} \Bigr)\Big|_{t=0}\Biggr)
= \sum_{n=0}^{\infty} \frac{t^n}{n!}[\text{ad}(iM_x)]^n(P).
\end{equation}
\end{proof}

\begin{lemma}\label{Lemma:AbsConv}
The power series for $e^{it\,\rm ad\it M_x}(P)$ in Lemma \ref{Lemma:PowerSeries} is absolutely convergent for all $t\in\mathbb{C}$ with respect to the operator norm.  Therefore  $e^{it\,\rm ad\it M_x}(P)$ is entire analytic on $\mathbb{C}$. 
\end{lemma}
\begin{proof}
First,
\begin{equation*}
\| [iM_x, P]\| = \| iM_xP - PiM_x \| \leq 2\|M_x\| \|P\|.
\end{equation*}
By the same reasoning,
\begin{equation*}
\begin{split}
\| [\text{ad}(iM_x)]^2(P) \| 
& =\| [iM_x, [iM_x, P]] \| \leq 2\|M_x\| \|[M_x, P]\|\\
& \leq 2^2 \|M_x\|^2\|P\|.
\end{split}
\end{equation*}
Again, by induction, we find
\begin{equation*}\label{Ineq:adMxPower}
\| [\text{ad}(iM_x)]^n(P) \|\leq 2^n\|M_x\|^n \|P\|. 
\end{equation*}

Going back to Equation (\ref{Eqn:PowerSeriesExpP}), we estimate the norm of the operator given by the power series:
\begin{equation}
\Bigg\|
\sum_{n=0}^{\infty} \frac{t^n}{n!}[\text{ad}(iM_x)]^n(P)
\Bigg\|
\leq
\|P\| \sum_{n=0}^{\infty} \frac{(2\|M_x\| \,|t|)^n}{n!} = \|P\|e^{2\|M_x\|\,|t|}.
\end{equation}
Notice that here, $t\in\mathbb{C}$, not just $\mathbb{R}$.  Therefore, by Theorem 10.6 in \cite{Rud87},  the map from $t \in \mathbb{C}$ to the operator $e^{it\,\rm ad\it M_x}(P)$ is entire analytic on $\mathbb{C}$.

Since $\|M_x\| = \textrm{diam}(\textrm{supp}(\mu))$, in our Bernoulli examples we can actually compute $\|M_x\| =\frac{\lambda}{1-\lambda}$.
\end{proof}

The power series for $e^{itM_x}Pe^{-itM_x}$ is only defined for $t\in\mathbb{R}$, but it agrees with the power series for $e^{it\,\rm ad\it M_x}(P)$ for all $t\in\mathbb{R}$.  Therefore, the operators are the same for all $t \in \mathbb{R}$, and $e^{itM_x}Pe^{-itM_x}$ has an entire analytic extension to the complex plane $\mathbb{C}$.  Now, we complete the proof of Theorem \ref{Thm:cGammaAnalytic} by showing that the spectral function $c_{\Gamma}$ also has an analytic extension to $\mathbb{C}$.

Consider the linear functional (state) $\phi$ on the bounded linear operators on $L^2(\mu)$ defined by
$ \phi(B) = \langle 1, B1\rangle_{L^2(\mu)}$, where  $1$  denotes the constant function whose value is always $1$.
By Cauchy-Schwarz, $\phi$ has norm $1$ and is therefore bounded.  By Lemma \ref{lem:cphi},  $c_{\Gamma}(t) = \phi(e^{itM_x}Pe^{-itM_x})$.  Applying $\phi$ to the operator $e^{itM_x}Pe^{-itM_x} = e^{it\,\rm ad\it M_x}(P)$ preserves analyticity in $t$, since $\phi$ does not disturb the $t$ variable.  Using the $\textrm{ad}$ representation of $e^{itM_x}Pe^{-itM_x}$, we have
\[ \phi\Bigl(e^{it\,\rm ad\it M_x}(P)\Bigr) 
= \sum_{n=0}^{\infty} \frac{t^n}{n!} \phi( [\rm ad \it(iM_x)]^n(P)).\]
This last expression is simply a complex function of the variable $t$.  

Finally, we claim that applying $\phi$ does not change the radius of convergence of the power series, since
\begin{eqnarray*} \sum_{n=0}^{\infty} \Big|\frac{t^n}{n!}\Big| 
\:\Big| \phi( [\rm ad \it(iM_x)]^n(P)) \Big|
&\leq & 
 \sum_{n=0}^{\infty} \Big|\frac{t^n}{n!}\Big| 
\:\Big\| [\rm ad \it(iM_x)]^n(P) \Big\|\\ &\leq& 
\sum_{n=0}^{\infty} \Big|\frac{t^n}{n!}\Big| 2^n \|M_x\|^n\|P\|.  \end{eqnarray*} \end{proof}

\subsection{The derivative of $c_{\Gamma}$}

We now give an estimate on the first derivative of the spectral function $c_\Gamma$.
\begin{prop}\label{Prop:cGammaDerivativeBdd}
Assume there exists a bounded operator $S:\ell^2(\Gamma) \rightarrow L^2(\mu)$ with $S(\delta_{\gamma}) = e_{\gamma}$ for all $\gamma\in \Gamma$.  Then $\frac{d}{dt} c_{\Gamma}(t)$ is   bounded on $\mathbb{R}$.
\end{prop}

\begin{proof}
We calculate the derivative of $c_{\Gamma}$, making use of the linear functional $\phi$ defined above.  We can now write
\[ c_{\Gamma}(t) = \phi(e^{itM_x}Pe^{-itM_x}).\]
Then we find \begin{equation*}
\begin{split}
\Big| \frac{d}{dt} c_{\Gamma}(t) \Big|
& = \Big| \phi\Bigl(\frac{d}{dt}e^{itM_x}Pe^{-itM_x}\Bigr)\Big|\\
& = \Big| \phi\Bigl(e^{itM_x}\textrm{ad}(iM_x)Pe^{-itM_x}\Bigr)\Big|\\
& \leq \|\phi\| \cdot 1 \cdot 2\|M_x\| \|P\|\cdot 1\\
& = 2 \|P\| \frac{\lambda}{1-\lambda},
\end{split}
\end{equation*}
which is independent of $t$.   
\end{proof}

We have now established that the spectral functions $ c_{\Gamma}$ are entire analytic and have bounded derivative on the real line.  It follows that they have at most exponential growth in the whole complex plane.  When spectral functions are non-constant, there is very little explicit information about them.  The corollary below serves to remedy this by giving  by applying the Weierstrass functions from the theory of entire functions to our $c_{\Gamma}$ and exploiting the results we obtained above.

\begin{cor}
Let $\mu$ be a measure satisfying the condition of Theorem \ref{Thm:cGammaAnalytic} and let $\Gamma \subset \mathbb{R}$ be a countable set generated recursively and having a constant scale factor, as in Equation (\ref{eqn:ngamma}).  Assume that $E(\Gamma)$ is orthogonal in $L^2(\mu)$ but that the spectral function $c_{\Gamma}$ is not the constant $1$ function.  Then there is an entire analytic function $G_{\Gamma}$ on $\mathbb{C}$ such that for all $t \in \mathbb{C}$,  $c_{\Gamma}$ satisfies \[c_{\Gamma}(t) = 1+G_{\Gamma}(t) \prod_{\gamma \in \Gamma}\Bigr(1+\frac{t}{\gamma} \Bigr) \exp\Biggr(\sum_{k=1}^{n(\gamma)} \frac{1}{k} \Bigr(-\frac{t}{\gamma} \Bigr)^k \Biggr), \] where the infinite product on the right hand side is absolutely convergent on every compact subset in $\mathbb{C}$.
\end{cor}
\begin{proof}
We define $c_{\Gamma}$ as in Equation (\ref{eqn:cdefn}) and let $D_{\Gamma}(t) = c_{\Gamma}(t)-1$.  Then the zeros of $D_{\Gamma}$ contain the set $-\Gamma$ since for a fixed $\gamma_0 \in \Gamma$, we have
\[ \widehat{\mu}(-\gamma_0+\gamma) = \left\{ \begin{matrix} 1 & \gamma = \gamma_0 \\ 0 & \gamma \in \Gamma \setminus \{\gamma_0\} \end{matrix} \right. . \]   This gives \[ c_{\Gamma}(-\gamma_0) = \sum_{\gamma \in \Gamma} \Bigr| \widehat{\mu}(-\gamma_0 + \gamma) \Bigr|^2 = 1.\]

It follows from Lemma \ref{Lemma:AbsConv}, Proposition \ref{Prop:cGammaDerivativeBdd} and \cite{JoPe98} that each element of $-\Gamma$ is an isolated zero of $D_{\Gamma}$ when $c_{\Gamma}$ is not the constant $1$ function.  As a result, we apply Theorem 15.10 in \cite{Rud87}.   For each fixed $\gamma$, the Weierstrass factors are
\[ E_{n(\gamma)}(t) = \Bigr(1+\frac{t}{\gamma} \Bigr) \exp\Biggr(\sum_{k=1}^{n(\gamma)} \frac{1}{k} \Bigr(-\frac{t}{\gamma} \Bigr)^k \Biggr).\]  
By \cite{Rud87} Theorem 15.10, we can write the entire analytic function \[ D_{\Gamma}(t) = G_{\Gamma}(t)\prod_{\gamma \in \Gamma}\Bigr(1+\frac{t}{\gamma} \Bigr) \exp\Biggr(\sum_{k=1}^{n(\gamma)} \frac{1}{k} \Bigr(-\frac{t}{\gamma} \Bigr)^k \Biggr), \] where $G_{\Gamma}$ is an entire analytic function on $\mathbb{C}$.  This completes the result.
\end{proof}
\section{New spectra for  $\mu_{\frac{1}{2n}}$}\label{sec:spectra}

We will now apply the results above to explore in detail the $L^2(\mu_{\frac{1}{2n}})$-basis properties of the dilated sets $p\Gamma(\frac{1}{2n})$, where $p$ is an odd integer and  $\Gamma(\frac{1}{2n})$ is as defined in Equation (\ref{eqn:ngamma}).  Recall  \cite{JoPe98}  that for each $n \in \mathbb{N}$,  $E(\Gamma(\frac{1}{2n}))$ is an ONB for $L^2(\mu_{\frac{1}{2n}})$.   In Section \ref{subsec:3gamma8}, we show that in the special case where $\lambda=\frac18$ and $p=3$, the set $E(3\Gamma(\frac18))$ is also an ONB.  In Section \ref{subsec:pgamman}, we find a condition on $p$ and $n$ which determines whether $E(p\Gamma(\frac{1}{2n}))$ is an ONB for  $L^2(\mu_{\frac{1}{2n}})$.

\subsection{$\mu_{\frac18}$ and the set $3\Gamma(\frac18)$}\label{subsec:3gamma8}

Let $\Gamma$ be the set defined in Equation (\ref{eqn:Gamma}).  Before proving that $E(3\Gamma(\frac{1}{8}))$ is an ONB with respect to the measure $\mu_{\frac18}$, we observe that $E(3\Gamma(\frac{1}{8}))$ is an orthonormal set in $L^2(\mu_{\frac{1}{8}})$. 

\begin{lemma}\label{Lemma:3Gamma18ON}
The set $E(3\Gamma(\frac{1}{8}))$ is orthogonal with respect to the measure $\mu_{\frac{1}{8}}$.
\end{lemma}
\begin{proof}  Recall from Equation (\ref{eqn:zero}) that the zero set of $\mu_{\frac{1}{8}}$, denoted $\mathcal{Z}_{\frac{1}{8}}$,  is the set
\begin{equation}\label{Eqn:ZeroSet18}
\mathcal{Z}_{\frac{1}{8}}
=
\Bigg\{ \frac{8^k(2\ell+1)}{4} : k\in \mathbb{N} \text{ and } \ell \in\mathbb{Z}\Biggr\}.
\end{equation}
If $\rho_1, \rho_2\in 3\Gamma(\frac{1}{8})$, then $
\widehat{\mu}_{\frac{1}{8}}(\rho_1 - \rho_2) = 0$
since we can write $\rho_1 - \rho_2 = 3(\gamma_1- \gamma_2)$, where $\gamma_1, \gamma_2\in \Gamma(\frac{1}{8})$.  However, $\gamma_1-\gamma_2\in \mathcal{Z}_{\frac{1}{8}}$ since $\Gamma(\frac{1}{8})$ represents an ONB for $L^2(\mu_{\frac{1}{8}})$.  By Equation (\ref{Eqn:ZeroSet18}), $\rho_1 - \rho_2\in \mathcal{Z}_{\frac{1}{8}}$ as well, since $3$ is odd, and therefore by Lemma \ref{lem:zero}, $e_{\rho_1}$ and $e_{\rho_2}$ are orthogonal in $L^2(\mu_{\frac18})$.
\end{proof}

\begin{lemma}\label{Lemma:BabyTransferResult}
Set
\[\Gamma = \Gamma\Bigl(\frac{1}{8}\Bigr) 
= \Biggl\{ \sum_{i=0}^{\rm{finite}} a_i\: 8^i \:\: :\:\: a_i\in \{0, 2\}\Biggr\},\]
and consider the spectral function defined by
\[ c_{3\Gamma}(t)=c_{3\Gamma (\frac{1}{8})}(t) = \sum_{\rho\in 3\Gamma} \Big| \widehat{\mu}_{\frac{1}{8}}(t + \rho)\Big|^2.\]
Then $c_{3\Gamma}$ satisfies the functional equation
\[ c_{3\Gamma}(t) 
= \cos^2\Bigl(\frac{\pi t}{4}\Bigr)
c_{3\Gamma}\Bigl(\frac{1}{8}t\Bigr) 
+ \sin^2  \Bigl(\frac{\pi t}{4}\Bigr)
c_{3\Gamma}\Bigl(\frac{1}{8}t + \frac{3}{4}\Bigr).\]
\end{lemma}
\begin{proof}Since $\Gamma = \{0,2\} + 8\Gamma$, we can write $3\Gamma = \{0, 6\} + 24\Gamma$.  Then we have
\begin{equation*}
\begin{split}
c_{3\Gamma}(t)
& = \sum_{\rho\in 24\Gamma} 
	\Big| \widehat{\mu}_{\frac{1}{8}}(t + \rho)\Big|^2
+ \sum_{\rho\in 6 + 24\Gamma} 
	\Big| \widehat{\mu}_{\frac{1}{8}}(t + \rho)\Big|^2\\
& = \sum_{\gamma \in \Gamma} 
	\Big| \widehat{\mu}_{\frac{1}{8}}(t + 24\gamma)\Big|^2
+ \sum_{\gamma \in \Gamma}\Big|
	\widehat{\mu}_{\frac{1}{8}}(t + 6 + 24\gamma)\Big|^2.
\end{split}
\end{equation*}
Now apply the identity $\widehat{\mu}(t) = \cos\Bigl(2 \pi \frac{1}{8} t\Bigr) \widehat{\mu}\Bigl(\frac{1}{8} t\Bigr)$:

\begin{equation*}
\begin{split}
c_{3\Gamma}(t)
& = \sum_{\gamma \in \Gamma} 
	\cos^2\Bigl(2\pi \frac{1}{8} (t + 24\gamma) \Bigr)
	\Big| \widehat{\mu}_{\frac{1}{8}}\Bigl(\frac{1}{8}(t + 24\gamma)\Bigr)\Big|^2\\
& \phantom{{=}}
+ \sum_{\gamma \in \Gamma}
	\cos^2\Bigl(2\pi \frac{1}{8} (t + 6 + 24\gamma) \Bigr)
	\Big|\widehat{\mu}_{\frac{1}{8}}
	\Bigl(\frac{1}{8}(t + 6 + 24\gamma)\Bigr)\Big|^2\\
& = \cos^2\Bigl(2\pi \frac{1}{8}t \Bigr)
	\sum_{\gamma \in \Gamma} 
	\Big| \widehat{\mu}_{\frac{1}{8}}
	\Bigl(\frac{1}{8}t + 3\gamma\Bigr)\Big|^2\\
& \phantom{{=}}
+ \cos^2\Bigl(2\pi \frac{1}{8}t + \frac{3\pi}{2}\Bigr)
\sum_{\gamma \in \Gamma}
	\Big|\widehat{\mu}_{\frac{1}{8}}
	\Bigl(\frac{1}{8}t + \frac{3}{4} + 3\gamma)\Bigr)\Big|^2.\\
\end{split}
\end{equation*}
But $\cos\Bigl(x + \frac{3\pi}{2}\Bigr) = \sin(x)$, and the sum over $\Gamma$ can be turned into a sum over $3\Gamma$.  Therefore
\begin{equation*}
c_{3\Gamma}(t) = \cos^2\Bigl(\frac{\pi}{4}t \Bigr)c_{3\Gamma}\Bigl(\frac{1}{8}t\Bigr)
+ \sin^2\Bigl(\frac{\pi}{4}t \Bigr)c_{3\Gamma}\Bigl(\frac{1}{8}t + \frac{3}{4}\Bigr).
\end{equation*}
\end{proof}

\begin{remark}\label{rmk:hadamard}Before moving further, we note Hadamard duality (as described in \cite{JoPe98}) arises in Lemma \ref{Lemma:BabyTransferResult}.  When we work with $\frac{1}{8}$, we work in the case $\frac{1}{2n}$ when $n = 4$.  Consider the set $\Gamma(\frac{1}{8})$.  Set $L = \{0,2\} = \{0, \frac{4}{2}\}$, the set of coefficients $\{a_i\}$ in Equation (\ref{eqn:Gamma}).  We can find at least two sets $B_1$ and $B_2$ such that $\{B_i, L, 8\}$ forms a Hadamard triple for $i = 1, 2$:
\begin{equation}
\label{Defn:B}
B_1 = \{ 0,2\} \textrm{ and } B_2 = \{-1, 1\}.
\end{equation}
For the set $3\Gamma(\frac{1}{8})$, $3L = \{0, 6\} = \{ 0, \frac{4\cdot 3}{2}\}$; then the sets $B_1$ and $B_2$ still have the property that $\{B_1, 3L, 8\}$ is a Hadamard triple, $i = 1,2$.  In Lemma \ref{Lemma:BabyTransferResult}, the last equation
\[ \cos^2\Bigl(\frac{\pi}{4}t \Bigr)c_{3\Gamma}\Bigl(\frac{1}{8}t\Bigr)
+ \sin^2\Bigl(\frac{\pi}{4}t \Bigr)c_{3\Gamma}\Bigl(\frac{1}{8}t + \frac{3}{4}\Bigr)\]
can be rewritten 
\begin{equation*}\label{Eqn:BabyTransferIFS}
\cos^2\Bigl(\frac{\pi}{4}t \Bigr)c_{3\Gamma}(\tau_0(t))
+ \sin^2\Bigl(\frac{\pi}{4}t \Bigr)c_{3\Gamma}(\tau_3(t)),
\end{equation*}
where 
\[ \tau_0(t) = \frac{1}{2n}t = \frac{1}{8}t\] and 
\[ \tau_p(x) = \tau_3(x) = \frac{1}{2n}t + \frac{p}{4} = \frac{1}{8}x + \frac{3}{4}.\]
Thus the IFS associated with $3L$ is $\{\tau_0, \tau_3\}.$
\end{remark}

In Lemma \ref{Lemma:Tcontractive}, we find that a specialization of the argument used in \cite[Lemma 5.1]{JoPe98} can be followed here for the case of the set $3\Gamma(\frac{1}{8})$ and the measure $\mu_{\frac{1}{8}}$.  

Define the transfer operator $T_{3L}$ to be
\begin{equation}\label{Eqn:T3L}
T_{3L}f(t) := \cos^2\Bigl(\frac{\pi}{4}t \Bigr) f\Bigl(\frac{1}{8}t\Bigr)
+ \sin^2\Bigl(\frac{\pi}{4}t \Bigr)f\Bigl(\frac{1}{8}t + \frac{3}{4}\Bigr).
\end{equation}
Then $T_{3L}$ is \textit{localized} on the interval $J:=[0, \frac{6}{7}]$---that is, if we restrict $f$ to $J$ (denoted $f|_{J}$), then
\begin{equation*}
T_{3L}(f|_{J}) = (T_{3L}f)|_{J}.
\end{equation*}
We can see this by checking that $\tau_0(J)\subset J$ and $\tau_3(J)\subset J$.
\begin{lemma}\label{Lemma:Tcontractive}
Let $\mathcal{K} = \{f: f\in C^1(J), f\geq 0, f(0) = 1\}$, and equip $\mathcal{K}$ with the norm $\max|f'|$. Define the transfer operator $T_{3L}$ as in Equation (\ref{Eqn:T3L}):
\begin{equation*}
T_{3L}f(t):=
\cos^2\Bigl(\frac{\pi}{4}t \Bigr) f\Bigl(\frac{1}{8}t\Bigr)
+ \sin^2\Bigl(\frac{\pi}{4}t \Bigr)f\Bigl(\frac{1}{8}t + \frac{3}{4}\Bigr).
\end{equation*}  
Then
$T$ is strictly contractive on $\mathcal{K}$.
\end{lemma}
\begin{proof}
Let $f\in C^1(J)$.  We calculate the derivative of $T_{3L}f(t)$ with respect to $t$:
\begin{equation*}
\begin{split}
 \frac{d}{dt}\Biggl(T_{3L}f(t) \Biggr)  & = 
\frac{\pi}{4}\sin\Bigl(\frac{\pi t}{2}\Bigr)
\Biggl(
	- f\Bigl(\frac{t}{8}\Bigr) 
	+f\Bigl(\frac{t}{8} + \frac{3}{4}\Bigr)
\Biggr)\\
& \phantom{{=}}+
\frac{1}{8}
\Biggl[
	\cos^2\Bigl(\frac{\pi t}{4}\Bigr)
	f'\Bigl(\frac{t}{8}\Bigr)
+
	\sin^2\Bigl(\frac{\pi t}{4}\Bigr)
	f'\Bigl(\frac{t}{8} + \frac{3}{4}\Bigr)
\Biggr].
\end{split}
\end{equation*}

Notice that the last term above is exactly $\frac18 (Tf')(t)$, and that $|(Tf')(t)| \leq \max|f'|$ for all $t$.    We can use this to estimate the absolute value of the derivative:

\begin{equation*}
\begin{split}
 \Bigg| 
	\frac{d}{dt}\Bigl(T_{3L}f(t) \Bigr)
\Bigg|
& \leq 
\frac{\pi}{4} 
\Bigg|
	\int_{\frac{t}{8} }^{\frac{t}{8}+ \frac{3}{4}} f'(s)\:ds
\Bigg|
+ 
\frac{1}{8}\max |f'|
\\
& \leq 
\Bigl(\frac{\pi}{4}\cdot \frac{3}{4} + \frac{1}{8}\Bigr)\max|f'|.
\end{split}
\end{equation*}
Therefore, $\max|(T_{3L}f)'| \leq \Bigl(\frac{3\pi}{16} + \frac{1}{8}\Bigr)\max|f'|$, and since $\frac{3\pi}{16} + \frac{1}{8} < 1$, $T_{3L}$ is strictly contractive.

To see that $T_{3L}$ maps $\mathcal{K}$ back into $\mathcal{K}$, we first calculate 
\[ T_{3L}f(0)  = \cos^2(0) f(0) + \sin^2(0)f\Bigl(\frac{3}{4}\Bigr) = f(0) = 1.\]
If $f\in C^1(J)$, then $T_{3L}f\in C^1(J)$ as well, since $\cos$ and $\sin$ are $C^{\infty}$ functions.
\end{proof} 

We are now ready to prove the main result. 

\begin{thm}\label{Thm:3Gamma18ONB}
The set $E(3\Gamma(\frac{1}{8}))$ is an orthonormal basis for $L^2(\mu_{\frac{1}{8}})$.
\end{thm}
\begin{proof}
We argue that the function $c_{3\Gamma}$ must be identically $1$.  Lemma \ref{Lemma:3Gamma18ON} showed that the set $E(3\Gamma(\frac18))$ is orthogonal, so by Lemma \ref{lem:cproperties}, we have $c_{3\Gamma}(t) \leq 1$ for all $t$ and hence the synthesis operator defined in Equation (\ref{Eqn:Sdeltagamma}) is bounded with operator norm $1$.   By Proposition \ref{Prop:cGammaDerivativeBdd}, we therefore know that $c_{3\Gamma}$ has bounded derivative, and furthermore $c_{3\Gamma}\in\mathcal{K}$.   

By Lemma \ref{Lemma:BabyTransferResult}, $c_{3\Gamma}$ is an eigenfunction of the transfer operator $T_{3L}$ with eigenvalue $1$.  Since $T_{3L}$ is strictly contractive with respect to the norm $\max|f'|$, we have
\[ \max |(c_{3\Gamma})'| = \max|(T_{3L}c_{3\Gamma})'| \leq  \Bigl(\frac{3\pi}{16} + \frac{1}{8}\Bigr)\max |(c_{3\Gamma})'|,\]
which implies that $c_{3\Gamma}$ must be constant.  Since $c_{3\Gamma}(0) = 1$, we must have $c_{3\Gamma
} \equiv 1$.
\end{proof}

\subsection{Families of ONBs for $L^2(\mu_{\frac{1}{2n}})$}\label{subsec:pgamman}
Now that we have proved that $E(3\Gamma(\frac{1}{8}))$ is an ONB for $L^2(\mu_{\frac{1}{8}}) = L^2(\mu_{\frac{1}{2\cdot 4}})$, we extend this result to Bernoulli measures $\mu_{\frac{1}{2n}}$ for  spectra of the form $p\Gamma(\frac{1}{2n})$ where $p\in 2\mathbb{N}+1$ and $n\in\mathbb{N}$.  Our result depends on the contractivity constant as in Theorem \ref{Thm:3Gamma18ONB}.  

Recall the general spectrum
\begin{equation*}
\Gamma\Bigl(\frac{1}{2n}\Bigr) 
= \Biggl\{ \sum_{i=0}^{\rm{finite}} a_i\: (2n)^i 
\:\: :\:\: a_i\in \Bigl\{0, \frac{n}{2}\Bigr\}\Biggr\}.
\end{equation*} 
Our generalized theorem is the following:
\begin{thm}
If $p \in 2\mathbb{N}+1$ such that $p < \frac{2(2n-1)}{\pi}$, then $E(p\Gamma(\frac{1}{2n}))$ is an ONB for $L^2(\mu_{\frac{1}{2n}})$.
\end{thm}
\begin{proof}
Fix $n\in\mathbb{N}$ and let $p\in 2\mathbb{N}+1$.  Set $\mu = \mu_{\frac{1}{2n}}$.  Because $p$ is odd, the family $E(p\Gamma(\frac{1}{2n}))$ is orthogonal in $L^2(\mu)$, just as in Lemma \ref{Lemma:3Gamma18ON}.  Define the spectral function $c$ by
\[ c(t) = c_{p\Gamma(\frac{1}{2n})}(t):=\sum_{\nu\in p\Gamma(\frac{1}{2n})}|\widehat{\mu}(t+\nu)|^2.\]
The structure given in Equation (\ref{Eqn:Gamma12n}) gives \[ p\Gamma\Bigl(\frac{1}{2n}\Bigr) = 2np\Gamma\Bigl(\frac{1}{2n}\Bigr) \cup \Bigl(\frac{np}{2} + 2np\Gamma\Bigl(\frac{1}{2n}\Bigr)\Bigr).\]
Therefore, we can pull apart the sum just as we did in Lemma \ref{Lemma:BabyTransferResult} to find that 
\[ c(t) = \cos^2\Bigl(\frac{\pi t}{n}\Bigr) c\Bigl(\frac{t}{2n}\Bigr) + \sin^2\Bigl(\frac{\pi t}{n}\Bigr)c\Bigl(\frac{t}{2n} + \frac{p}{4}\Bigr).\]
The only difference in the calculation is that $\cos\Bigl(t + \frac{p\pi}{2}\Bigr) = -\sin(t)$ when $p \equiv 1$ mod $4$ and $\cos\Bigl(t + \frac{p\pi}{2}\Bigr) = \sin(t)$ when $p\equiv 3$ mod $4$.

Therefore, there is a transfer operator $T_{pL}$ which has $c$ as an eigenfunction:
\begin{equation}\label{Eqn:TransferpL}
\begin{split}
T_{pL}f(t) 
& = \cos^2\Bigl(\frac{\pi t}{n}\Bigr) f\Bigl(\frac{t}{2n}\Bigr) + \sin^2\Bigl(\frac{\pi t}{n}\Bigr) f\Bigl(\frac{t}{2n} + \frac{p}{4}\Bigr)\\
& = \cos^2\Bigl(\frac{\pi t}{n}\Bigr) f(\tau_0(t)) + \sin^2\Bigl(\frac{\pi t}{n}\Bigr) f(\tau_p(t)).
\end{split}
\end{equation}
Following the derivative calculations in Lemma \ref{Lemma:Tcontractive}, we have 
\begin{eqnarray*} (T_{pL}f)'(t) &=&  \frac{\pi}{n}\sin\Bigl(\frac{2\pi t}{n}\Bigr)\Bigl[f\Bigl(\frac{t}{2n}+\frac{p}{4}\Bigr) - f\Bigl(\frac{t}{2n}\Bigr)\Bigr] + \frac{1}{2n} (T_{pL}f')(t)\\ &=&  \frac{\pi}{n}\sin\Bigl(\frac{2\pi t}{n}\Bigr)\int_{\frac{t}{2n}}^{\frac{t}{2n} + \frac{p}{4}} f'(s) \,\textrm{d}s + \frac{1}{2n} (T_{pL}f')(t) \end{eqnarray*}

Using the same computations, we have an estimate on the derivative of $T_{pL}f$:
\[ \Big| (T_{pL}f)'(t)\Big| \leq \frac{\pi}{n}\frac{p}{4} \max|f'| + \frac{1}{2n} \max|f'|. \]

Thus, we find that $T_{pL}$ is strictly contractive when 
\[ \frac{\pi}{n}\cdot \frac{p}{4} + \frac{1}{2n} < 1,\]
or 
\begin{equation}\label{eqn:pn} p < \frac{2(2n-1)}{\pi}.\end{equation}
\end{proof}

For example, when $n=2$ and $n = 3$, only $p=1$ satisfies the inequality above.  When $n = 4$, $p = 1$ \cite{JoPe98} and $p=3$ (Lemma \ref{Lemma:Tcontractive}) work, but $T_{5L}$ does not have a contractivity constant less than $1$.  When $n = 5$, $p = 1$ \cite{JoPe98}, $p=3$, and $p=5$ satisfy the inequality, and for $n=6$ we have $p=1, 3, 5, $ and $7$ all yielding different spectral pairs $\bigl(\mu_{\frac{1}{12}}, p\Gamma(\frac{1}{12})\bigr)$. 

\begin{remark}\label{rmk:examples}We should emphasize here that the lists of orthogonal exponentials $E(S)$ we have described in this section for the various Bernoulli measures are not exhaustive.  We have only considered those of the form $p\Gamma(\frac{1}{2n})$ where $p$ is odd and satisfies the condition in Equation (\ref{eqn:pn}).   There are a variety of sets of exponentials $E(S)$  known to be ONBs with respect to a particular fractal measure, but for which
there is no known $C^1$-contractive transfer operator that can account for it.  For example, in the case where $n=2$ and  $p=5$, it is known \cite{DJ09} that $5\Gamma(\frac{1}{4})$ is a spectrum for $L^2(\mu_{\frac{1}{4}})$ even though $p=5$ is certainly larger than $\frac{6}{\pi}$.  Further, this example gives an ONB for all spectra of the form $5^k\Gamma(\frac{1}{4}), k \in \mathbb{N}$, so the odd number $p$ can get arbitrarily large and still yield an ONB.

On the other hand, there are examples for which $p$ is larger than the inequality in (\ref{eqn:pn}) and for which the sets $E(p\Gamma(\frac{1}{2n}))$ are orthogonal but not complete, i.e. their span is not dense in $L^2(\mu_{\frac{1}{2n}})$.  One such example is for given $n \in \mathbb{N}$,  let $p=2n-1$.  Then $E(p\Gamma(\frac{1}{2n}))$ is not complete in $L^2(\mu_{\frac{1}{2n}})$, and in fact, the exponential $e_{-\frac{n}{2}}$ is in the orthogonal complement of the span of $E(p\Gamma(\frac{1}{2n}))$.  To see this, we take the inner product of $e_{p\gamma}$ for $\gamma \in \Gamma(\frac{1}{2n})$ and $e_{-\frac{n}{2}}$:
\[ \langle e_{-\frac{n}{2}}, e_{p\gamma} \rangle = \widehat{\mu}_{\frac{1}{2n}}\Big(\frac{n}{2} + (2n-1)\gamma\Big).\]  If $\gamma = \sum_{i=0}^p a_i (2n)^i$, recalling that $a_i \in \{0, \frac{n}{2}\}$, then we see that \[ \frac{n}{2} + 2n\gamma-\gamma = \frac{n}{2} + \sum_{i=0}^p a_i (2n)^{i+1} -\sum_{i=0}^p a_i (2n)^i\]  Then $\gamma' = \frac{n}{2} + \sum_{i=0}^p a_i (2n)^{i+1}$ is another element of $\Gamma(\frac{1}{2n})$.  Since the difference of any two elements in $\Gamma(\frac{1}{2n})$ is in the zero set of $\widehat{\mu}_{\frac{1}{2n}}$, the inner product of the exponentials is indeed zero.

In other cases where the odd integer $p$ is greater than  $\frac{2(2n-1)}{\pi}$, it is not known whether $p\Gamma$ is a spectrum.  For example, numerical approximation of the spectral function $c_{5\Gamma}$ when $n=4$ suggests that $5\Gamma(\frac{1}{8})$ is not a spectrum for $\mu_{\frac{1}{8}}$, but a proof of this fact is not known.
\end{remark}

\begin{remark}These cases where the spectrum in one spectral pair $(\mu, \Gamma)$ can be dilated by an odd integer $p$ to yield a new spectral pair $(\mu, p\Gamma)$ are surprising.  This phenomenon cannot occur in classical Fourier analysis, for example, where 
  the spectral pair consists of Lebesgue measure $\nu$ on a period interval and the integers $\mathbb{Z}$.  \end{remark}
  
\section{Maximal orthogonal families which are not ONBs}\label{Sec:MaximalNotONB}

In this section, we discuss two additional results which are motivated by the main part of this paper.  They both apply to measures $\mu_{\lambda}$ for the case where $\lambda = \frac{q}{2n}$, with $q$ odd and not equal to $1$, i.e. the cases where $\lambda$ is not the reciprocal of an integer.  There are no known examples for such $\lambda$ for which $\mu_{\lambda}$ is spectral.    We concentrate in particular on the case where $\lambda = \frac{3}{8}$. 

First, we explore the maximality of an orthogonal collection of exponentials $E(\Gamma(\frac{1}{8}))$ which may not be an ONB for $L^2(\mu_{\frac{3}{8}})$.  Next, we develop a chain of orthogonal families connected by a transfer operator, and find that if any one family in the chain is an ONB, then they all are ONBs.  We found this to be an interesting structure, although we have not yet found an example where the structure actually produces ONBs for a particular measure.
 
\subsection{Maximality results}\label{Subsec:Maximality}

In this first case, let $\lambda = \frac38$ and \[\Gamma = \Gamma\Bigl(\frac18\Bigr) = \Bigl\{ \sum_{i=0}^{\text{finite}} a_i8^i\,:\, a_i \in \{0,2\} \Bigr\} = \{ 0, 2, 16, 18, 128, \dots\}. \]  We have already shown that the set $E(\Gamma)$ is an orthonormal basis with respect to the measure $\mu_{\frac18}$.  We can also show that it is orthogonal with respect to $\mu_{\frac38}$.  To see this, recall from Equation (\ref{eqn:zero}) that zero set for $\widehat{\mu}_{\frac38}$ is given by
 \begin{equation}\label{eqn:8-form} \mathcal{Z}_{\frac{3}{8}} = \Bigr\{ \frac{(2\ell+1) 8^k}{4 \cdot 3^k}\,:\, \ell \in \mathbb{Z}, k \geq 1 \Bigr\}.\end{equation} 
Given $\gamma, \gamma' \in \Gamma$, we know $\gamma-\gamma'$ in $\mathcal{Z}_{\frac18}$ (see Equation (\ref{Eqn:ZeroSet18})).  We use the corresponding value of $k$ to write $\gamma-\gamma'$ in the form for $\mathcal{Z}_{\frac{3}{8}}$.    Specifically, we have \[ \gamma-\gamma' = \frac{ 8^k(2\ell +1)}{4} = \frac{8^k\cdot 3^k(2\ell+1)}{4 \cdot 3^k}.\] Since $3^k(2\ell+1)$ is also odd, we have $\gamma - \gamma' \in \mathcal{Z}_{\frac{3}{8}}$.

It is suspected, but not known, that the set $E(\Gamma)$ is not an ONB for $\mu_{\frac38}$.  One might hope to use Theorem 3.5 in\cite{DHSW09} to argue that the Hausdorff dimension of the attractor set $X_{\frac{3}{8}}$ will differ from that of $X_{\frac18}$, and therefore cannot equal the Beurling dimension of $\Gamma(\frac18)$.   This would require verification of the technical condition given in Equation (3.3) of \cite{DHSW09}, which we cannot do in this example.  We can, however, show that $E(\Gamma)$ is maximally orthogonal in $L^2(\mu_{\frac38})$, i.e. there is no exponential function orthogonal to every element of $E(\Gamma)$. 
  
\begin{thm}\label{Thm:38Maximality} Given any $t \in \mathbb{R} \setminus \Gamma$, there exists some $\gamma \in \Gamma$ such that $\langle e_t, e_{\gamma} \rangle \neq 0$.  
 \end{thm}
\begin{proof}  
  We will show that $\langle e_t, e_{\gamma} \rangle \neq 0$ by showing that $\widehat{\mu}_{\frac38}(t - \gamma) \neq 0$, i.e. that $t - \gamma \notin \mathcal{Z}_{\frac{3}{8}}$.   First, observe that for any $t \notin \mathcal{Z}_{\frac{3}{8}}$, we can take $\gamma = 0$.  Therefore, we can restrict our attention to elements of $\mathcal{Z}_{\frac{3}{8}}$:
  We can also write elements of $\mathcal{Z}_{\frac{3}{8}}$ in the most reduced form as  \begin{equation}\label{eqn:2-form} \Bigr\{ \frac{(2m+1) 2^{3k-2}}{3^p}\Bigr\}, \end{equation} where $n \in \mathbb{Z}$, $k \geq 1$, and either $p=0$ or $p >0$ and $2n+1$ is not divisible by $3$.
  
Since we are interested in values $t$ which are not in $\Gamma$, we need to understand how the elements of $\Gamma$ look as elements of $\mathcal{Z}_{\frac{3}{8}}$.  Since elements of $\Gamma$ are all even integers, we know that $p=0$.  Moreover, we show that for $\gamma = \sum_{i=0}^n a_i8^i$,  if $i'$ is the first index such that $a_i \neq 0$, we take $k = i'+1$.  

\begin{eqnarray*} \gamma =  \sum_{i=i'}^n a_i8^i &=& 8^{i'}(2 +  \sum_{i=k}^n a_i8^{i-i'}) \\ &=&  \frac14 \cdot \frac{8^{i'}}{3^{i'}} \cdot 4(2+ \sum_{i=k}^n a_i8^{i-i'}) \cdot 3^{i'} \\  &=&  \frac14 \cdot \frac{8^{k}}{3^{k}} \cdot (1+ \sum_{i=k}^n \frac{a_i}{2}8^{i-i'}) \cdot 3^{k} = 2^{3k-2}(1+ \sum_{i=k}^n \frac{a_i}{2}8^{i-i'})
\end{eqnarray*}
  
Observe that the quantity in parentheses in the last $3$ lines above is an odd integer.  In fact, we know that when $\gamma \in \Gamma$ is written in the $\mathcal{Z}_{\frac{3}{8}}$ form, the odd integer in the numerator is of the form $3^k(1 + b_18 + b_2 8^2 + \cdots + b_q8^q)$ with $b_i \in \{0,1\}$ and where again, $k=i'+1$.

\noi (\textbf{Case 1: $p=0$)} \;  First, consider $t$ in the form of Equation (\ref{eqn:2-form}) for which $p=0$, i.e. $t$ is an even integer.  Next, we factor powers of $2$  and write the remaining odd factor in its base-8 expansion:   $t = 2^d(b_0 + \sum_{i=1}^m  b_i8^i)$, where $b_0$ is $1, 3, 5$, or $7$ and $b_i \in \{0, 1, \ldots, 7\}$ for each $i=1, 2, \ldots, n$.  Since $t \in \mathcal{Z}_{\frac{3}{8}}$, we know that the power $d$ must be of the form $3\ell-2$, i.e. $d \equiv 1 (\bmod \,3)$.    Let $i'$ be the smallest index for which $b_i \notin \{0,1\}$.  Using the values $d, i'$, we construct the elements of $\Gamma$: \[ \gamma_1 = 2^d, \quad  \gamma_2 =  2^d\sum_{i=0}^{i'-1} b_i8^i, \quad \gamma_3 = 2^d\sum_{i=0}^{i'} b_i8^i, \] where for $\gamma_2$ and $\gamma_3$ we require $i'>0$ and, thereby, $b_0=1$.   

If $i' = 0$, then \[t-\gamma_1 = 2^d \Bigr[ (b_0-1) + \sum_{i=1}^n b_i  8^i \Bigr] \] where $b_0 = 3, 5, $ or $7$.  The quantity in square brackets is even, perhaps divisible by $4$, but not divisible by $8$.  Therefore, when $t-\gamma_1$ is written in the form of Equation (\ref{eqn:2-form}), the power of $2$ is not congruent to $1$ modulo $3$ and hence $t-\gamma$ is not in $\mathcal{Z}_{\frac{3}{8}}$.  

If $b_0 =1$, i.e. $i'>0$, then if $b_{i'}$ is odd, consider the difference $t-\gamma_3$:
\[ t-\gamma_3 = 2^d 8^{i'}\Bigr[(b_{i'} - 1) + \sum_{i=i'+1}^n b_i 8^{i-i'}\Bigr] = 2^{3\ell-2}\Bigr[(b_{i'} - 1) + \sum_{i=i'+1}^n b_i 8^{i-i'}\Bigr] .\]  Since $b_{i'}$ is odd, the quantity in square brackets is even, possibly divisible by $4$ but not divisible by $8$.  If all the powers of $2$ are factored out, the new power of $2$ is not congruent to $1$ modulo $3$, and thus $t-\gamma_3$ is not in $\mathcal{Z}_{\frac{3}{8}}$.  

If $b_{i'}$ is even, look at the difference $t-\gamma_2$.  We have

\[ t-\gamma_2 = 2^d\Bigr[ b_i' + \sum_{i=i'+1}^{n} b_i8^{i'}\Bigr]. \]  Since $b_{i'}$ is even (is 2, 4, or 6), we can factor out 1 or 2 powers of 2, but not 3, and again have the wrong power of $2$, which proves $t-\gamma_2 \notin \mathcal{Z}_{\frac{3}{8}}$.  

\noi \textbf{(Case 2: $p=1$) }\;  This case works very similarly to Case 1, since $t$ is of the form \[ t = \frac{2^d(2m+1)}{3}\] where $3$ does not divide $2n+1$ and $d \equiv 1 (\bmod 3)$.    Writing the odd factor $2m+1$ in the numerator in its base-8 expansion gives: 

$t = \frac{2^d}{3}(b_0 + \sum_{i=1}^n b_i8^i)$ where $b_0$ is $1, 3, 5$, or $7$ and $b_i \in \{0, 1, \ldots, 7\}$ for each $i=1, 2, \ldots, p$.    This time, let $i'$ be the smallest index for which $b_i \notin \{0,3\}$.  Then, take the elements of $\Gamma$: \[ \gamma_1 = 2^d, \quad  \gamma_2 =  2^d\sum_{i=0}^{i'-1} \frac{b_i}{3}8^i, \quad \gamma_3 = 2^d\sum_{i=0}^{i'} \frac{b_i}{3}8^i \quad (b_0 = 3). \]

If $i' = 0$, then \[t-\gamma_1 = \frac{2^d}{3} \Bigr[(b_0-3) + \sum_{i=1}^n b_i 8^i\Bigr],\] where $b_0 = 1, 5,$ or $7$.   As before, the quantity within the square brackets is even and may be divisible by 4, but not by 8, which means the power of $2$ when $t-\gamma_1$ is written in Equation (\ref{eqn:2-form}) form will be wrong and hence $t-\gamma_1 \notin \mathcal{Z}_{\frac{3}{8}}$.

If $i'>0$ and $b_{i'}$ is odd (1, 5, or 7), then 

\[ t-\gamma_3 = \frac{2^d}{3} 8^{i'}\Bigr[(b_{i'} - 3) \sum_{i=i'+1}^n b_i 8^{i-i'}\Bigr] =  \frac{2^{3\ell-2}}{3}\Bigr[(b_{i'} - 3) \sum_{i=i'+1}^p b_i 8^{i-i'}\Bigr] .\]  

If $i'>0$ and $b_{i'}$ is even (2, 4, or 6), then 

\[ t-\gamma_2 = \frac{2^d}{3}\Bigr[ b_i' + \sum_{i=i'+1}^n b_i8^{i'}\Bigr]. \]

In both instances, the quantity in square brackets is divisible by one or two powers of $2$, but not three.  Thus, when all of the powers of $2$ are factored out, the resulting power of $2$ is not congruent to $1 (\bmod \,3)$, so the differences are not in $\mathcal{Z}_{\frac{3}{8}}$.

\noi \textbf{(Case 3: $p>1$)}\;  If $p>1$, we take $\gamma = 2$.   Recall that $d \geq 3p-2 > 1$.  We have  
\[ t-\gamma =  \frac{2^d [1+ \sum_{i=1}^{n} b_i8^{i}] - 2 \cdot 3^p}{3^p}. \]  Since the odd integer $1+ \sum_{i=1}^{n} b_i8^{i}$ is not divisible by 3, we see that the numerator is not divisible by $3$ or by $8$.  When we put this into the reduced form of Equation (\ref{eqn:2-form}), we get \[ t-\gamma = \frac{2[2^{d-1}(1+\sum_{i=1}^{n} b_i8^{i})-3^p]}{3^p}.\]  The numerator is divisible by 2, but no higher powers of 2, which proves that $t-\gamma \notin \mathcal{Z}_{\frac{3}{8}}$.  
\end{proof}

\subsection{The transfer operators for $\lambda=\frac{3}{8}$}\label{Subsec:TransferT}
 
We have observed before the structure of the set $\Gamma(\frac{1}{8})$:
\begin{equation}\label{Eqn:StructureOfGamma}
\Gamma\Bigl(\frac{1}{8}\Bigr) = \{0, 2\} + 8 \Gamma\Bigl(\frac{1}{8}\Bigr).
\end{equation}

We define a spectral function similar to those defined earlier, but now using the measure $\mu_{\frac38}$ and summation over the set $\Gamma(\frac18)$.
\begin{equation}
c_{\Gamma(\frac{1}{8}), \frac{3}{8}}(t) := \sum_{\gamma\in\Gamma(\frac{1}{8})} \Big| \widehat{\mu}_{\frac{3}{8}}(t+\gamma)\Big|^2.
\end{equation}
We know that $c_{\Gamma(\frac{1}{8}), \frac{3}{8}}(t) \leq 1$ for all $t\in\mathbb{R}$ and that $c_{\Gamma(\frac{1}{8}), \frac{1}{8}}(t) = 1$ for all $t\in\mathbb{R}$ because $\Gamma(\frac{1}{8})$ corresponds to an ONB for $\mu_{\frac{1}{8}}$.  We also know that 
\begin{equation*}
c_{\Gamma(\frac{1}{8}), \frac{3}{8}}(0)  = 1
\end{equation*} 
because $\widehat{\mu}_{\frac{3}{8}}(0+0) = 1$ and $\widehat{\mu}_{\frac{3}{8}}(0+\gamma) = 0$ for $\gamma\in\Gamma(\frac{1}{8}) \backslash \{0\}$.

Using Equation (\ref{Eqn:StructureOfGamma}), we can write
\begin{equation*}
\begin{split}
c_{\Gamma(\frac{1}{8}), \frac{3}{8}}(t) 
& = \sum_{\gamma\in\Gamma(\frac{1}{8})} \Big| \widehat{\mu}_{\frac{3}{8}}(t+\gamma)\Big|^2\\
& = \sum_{\gamma\in\Gamma(\frac{1}{8})} \Big| \widehat{\mu}_{\frac{3}{8}}(t+ 8\gamma)\Big|^2 + \sum_{\gamma\in\Gamma(\frac{1}{8})} \Big| \widehat{\mu}_{\frac{3}{8}}(t+2+8\gamma)\Big|^2
\end{split}
\end{equation*}

We use Equations (\ref{eqn:lambdamuhat}) and (\ref{eqn:muhat}) for $\lambda = \frac38$ to find
\begin{equation}\label{Eqn:KeyIdentity38}
\widehat{\mu}_{\frac{3}{8}}(t) 
= \cos\Bigl(2\pi \Bigl(\frac{3t}{8}\Bigr)\Bigr)\widehat{\mu}_{\frac{3}{8}}\Bigl(\frac{3t}{8}\Bigr).
\end{equation}

This gives a formulation of the function $c_{\Gamma(\frac{1}{8}), \frac{3}{8}}$:
\begin{equation*}
\begin{split}
c_{\Gamma(\frac{1}{8}), \frac{3}{8}}(t) 
& = \sum_{\gamma\in\Gamma(\frac{1}{8})} 
      \cos^2\Bigl(2\pi \Bigl(\frac{3t}{8} + 3\gamma\Bigr)\Bigr)
          \Big| \widehat{\mu}_{\frac{3}{8}}
          \Bigl(\frac{3t}{8}+ 3\gamma\Bigr)\Bigr)\Big|^2  \\
&\phantom{{= + }}    + \sum_{\gamma\in\Gamma(\frac{1}{8})}
      \cos^2\Bigl(2\pi \Bigl(\frac{3t}{8} + \frac{3}{4} + 3\gamma\Bigr)\Bigr)
      \Big| \widehat{\mu}_{\frac{3}{8}}\Bigl(\frac{3t}{8}
      +\frac{3}{4}+3\gamma\Bigr)\Big|^2.\\
\end{split}
\end{equation*}
But $\gamma\in\mathbb{Z}$, so
\begin{equation*}
\begin{split}
c_{\Gamma(\frac{1}{8}), \frac{3}{8}}(t) 
& = \sum_{\gamma\in\Gamma(\frac{1}{8})} 
      \cos^2\Bigl(2\pi \Bigl(\frac{3t}{8}\Bigr)\Bigr)
          \Big| \widehat{\mu}_{\frac{3}{8}}
          \Bigl(\frac{3t}{8}+ 3\gamma\Bigr)\Big|^2  \\
&\phantom{{= + }}    + \sum_{\gamma\in\Gamma(\frac{1}{8})}
      \sin^2\Bigl(2\pi \Bigl(\frac{3t}{8}\Bigr)\Bigr)
      \Big| \widehat{\mu}_{\frac{3}{8}}\Bigl(\frac{3t}{8}
      +\frac{3}{4}+3\gamma\Bigr)\Big|^2\\
 &= 
      \cos^2\Bigl(2\pi \Bigl(\frac{3t}{8}\Bigr)\Bigr)
c_{3\Gamma(\frac{1}{8}), \frac{3}{8}}\Bigl(\frac{3t}{8}\Bigr)  + 
      \sin^2\Bigl(2\pi \Bigl(\frac{3t}{8}\Bigr)\Bigr)
c_{3\Gamma(\frac{1}{8}), \frac{3}{8}}\Bigl(\frac{3t}{8} + \frac{3}{4}\Bigr).
\end{split}
\end{equation*}

As we did in earlier sections, we define the transfer operator $T$ to be
\begin{equation}\label{Defn:TransferT}
(Tf)(t) = \cos^2\Bigl(2\pi \Bigl(\frac{3t}{8}\Bigr)\Bigr)f\Bigl(\frac{3t}{8}\Bigr) + \sin^2\Bigl(2\pi \Bigl(\frac{3t}{8}\Bigr)\Bigr)f\Bigl(\frac{3t}{8} + \frac{3}{4}\Bigr).
\end{equation}

Now, if we define $f_1(t) = c_{3\Gamma(\frac{1}{8}), \frac{3}{8}}(t)$, then
\begin{equation}\label{Eqn:f0Tf1}
f_0(t) = Tf_1(t).
\end{equation}

We can iterate the previous calculation.  We define
\begin{equation*} 
f_k(t) = \sum_{\gamma\in 3^k\Gamma(\frac{1}{8}), \frac{3}{8}} 
       |\widehat{\mu}_{\frac{3}{8}}(t + \gamma)|^2
\end{equation*} 
and then see how $f_{k+1}$ and $f_k$ are related.  As before, write $\Gamma = 8\Gamma + (8\Gamma + 2)$.  Then 

\begin{equation*}
\begin{split}
f_k(t) 
& = \sum_{\gamma\in 3^k\Gamma(\frac{1}{8}), \frac{3}{8}} 
       |\widehat{\mu}_{\frac{3}{8}}(t + \gamma)|^2
= \sum_{\gamma\in \Gamma(\frac{1}{8}), \frac{3}{8}} 
       |\widehat{\mu}_{\frac{3}{8}}(t + 3^k\gamma)|^2\\
& = \sum_{\gamma\in \Gamma(\frac{1}{8}), \frac{3}{8}} 
       |\widehat{\mu}_{\frac{3}{8}}(t + 3^k8\gamma)|^2
    + \sum_{\gamma\in \Gamma(\frac{1}{8}), \frac{3}{8}} 
       |\widehat{\mu}_{\frac{3}{8}}(t + 3^k(8\gamma + 2))|^2.\\
\end{split}
\end{equation*}

Again, apply Equation (\ref{Eqn:KeyIdentity38}) to the equation for $f_k$: 

\begin{equation*}
\begin{split}
f_k(t) 
& = \sum_{\gamma\in \Gamma(\frac{1}{8}), \frac{3}{8}}
    \cos^2\Bigl(
                \frac{2\pi 3t}{8} + \frac{2\pi 3^{k+1}8\gamma}{8} 
          \Bigr)
    \Big|\widehat{\mu}_{\frac{3}{8}}
         \Bigl(
               \frac{3t}{8} + \frac{3^{k+1}8\gamma}{8} 
         \Bigr) \Big|^2\\
& \phantom{{==}} +
\sum_{\gamma\in \Gamma(\frac{1}{8}), \frac{3}{8}}
    \cos^2\Bigl( 
               \frac{2\pi 3t}{8} + \frac{2\pi 3^{k+1}8\gamma}{8} + \frac{2\pi 3^{k+1}2}{8}
          \Bigr) \\
          & \phantom{{==   }}\phantom{{==}} \cdot
    \Big|\widehat{\mu}_{\frac{3}{8}}
         \Bigl(
              \frac{3t}{8} + \frac{3^{k+1}8\gamma}{8} + \frac{3^{k+1}2}{8}
         \Bigr)\Big|^2\\
& = \sum_{\gamma\in \Gamma(\frac{1}{8}), \frac{3}{8}}
    \cos^2\Bigl(
                \frac{2\pi 3t}{8} 
          \Bigr)
    \Big|\widehat{\mu}_{\frac{3}{8}}
         \Bigl(
               \frac{3t}{8} + 3^{k+1}\gamma 
         \Bigr) \Big|^2\\
& \phantom{{==}} +
\sum_{\gamma\in \Gamma(\frac{1}{8}), \frac{3}{8}}
    \cos^2\Bigl( 
               \frac{2\pi 3t}{8} + \frac{\pi 3^{k+1}}{2}
          \Bigr)
    \Big|\widehat{\mu}_{\frac{3}{8}}
         \Bigl(
              \frac{3t}{8} + 3^{k+1}\gamma + \frac{3^{k+1}}{4}
         \Bigr)\Big|^2\\
\end{split}
\end{equation*}
\normalsize

Since $3^{k+1}$ is always odd, the second $\cos^2$ turns into a $\sin^2$ as before.  Therefore
\begin{equation*}
\begin{split}
f_k(t) 
& = \sum_{\gamma\in \Gamma(\frac{1}{8}), \frac{3}{8}}
    \cos^2\Bigl(
                \frac{2\pi 3t}{8} 
          \Bigr)
    \Big|\widehat{\mu}_{\frac{3}{8}}
         \Bigl(
               \frac{3t}{8} + 3^{k+1}\gamma 
         \Bigr) \Big|^2\\
& \phantom{{==}} +
\sum_{\gamma\in \Gamma(\frac{1}{8}), \frac{3}{8}}
    \sin^2\Bigl( 
               \frac{2\pi 3t}{8}
          \Bigr)
    \Big|\widehat{\mu}_{\frac{3}{8}}
         \Bigl(
              \frac{3t}{8} + 3^{k+1}\gamma + \frac{3^{k+1}}{4}
         \Bigr)\Big|^2\\
\end{split}
\end{equation*}
If we define the transfer operator $T_{k}$ to be
\begin{equation}\label{eqn:Tk}
T_{k}f(t)
:= \cos^2\Bigl(\frac{2\pi 3t}{8}\Bigr)f\Bigl(\frac{3t}{8}\Bigr)
+ \sin^2\Bigl(\frac{2\pi 3t}{8}\Bigr)f\Bigl(\frac{3t}{8}+\frac{3^{k+1}}{4}\Bigr),
\end{equation}
then
\begin{equation}
T_{k}f_{k+1}(t) = f_k(t).
\end{equation}
Observe that the transfer operators $T_k$ are dependent on $k$.  Our first example, then, was $T = T_0$, since $Tf_1 = f_0$.

\subsection{Implications for ONBs}

With the use of the family of transfer operators $T_k$ defined in Equation (\ref{eqn:Tk}), we explore possible sets of Fourier frequencies for
$\mu_{\frac{3}{8}}$.  First, note that $T$ maps the constant function $1$ to $1$:
\begin{equation*}
(T1)(t) = \cos^2\Bigl(2\pi \Bigl(\frac{3t}{8}\Bigr)\Bigr)\cdot 1 + \sin^2\Bigl(2\pi \Bigl(\frac{3t}{8}\Bigr)\Bigr)\cdot 1 = 1.
\end{equation*}

Suppose on the other hand that $Tf \equiv 1$ implied $f \equiv 1$ as well.  In this case, if we knew that $f_0 \equiv 1$, we would also have $f_1 \equiv 1$.  In other words, we  would  have two different ONBs for $L^2(\mu_{\frac{3}{8}})$.   In the next proposition, we show that this does hold, and in fact, can be extended for all $k$.  

\begin{proposition}\label{Prop:ChainofONBs}
If for some positive $k_0$, the set $E(3^{k_0}\Gamma)$ is an orthonormal basis for $L^2(\mu_{\frac38})$, then for every $k \geq 0$, the set $E( 3^k \Gamma)$ is an ONB.  \end{proposition}

\begin{proof}
As above, we have shown that $T_{k-1}f_{k} = f_{k-1}$ for all $k >0$.  If $f_k \equiv 1$, then it is clear from the definition of each $T_k$ that $f_{k-1} \equiv 1$ as well.  

Using the relationship $T_k f_{k+1} =f_k$,  suppose that $f_k \equiv 1$.  When $T_kf_{k+1}(t)$ is considered as the inner product of the 2-dimensional real vectors 

\begin{equation*} \left[\begin{matrix} \cos \Bigr(\frac{2\pi3t}{8} \Bigr) \\ \sin  \Bigr(\frac{2\pi3t}{8} \Bigr) \end{matrix}\right] \qquad \textrm{and} \qquad \left[\begin{matrix} \cos  \Bigr(\frac{2\pi3t}{8} \Bigr) f_{k+1}  \Bigr(\frac{3t}{8} \Bigr) \\ \sin  \Bigr(\frac{2\pi3t}{8} \Bigr) f_{k+1} \Bigr(\frac{3t}{8} + \frac{3^{k+1}}{4}\Bigr) \end{matrix}\right], \end{equation*}
the Cauchy-Schwarz inequality yields for every $t$

\begin{eqnarray*}  1 &=& T_kf_{k+1}(t) \\&\leq&  1 \cdot  \Bigr[ \cos^2\Big(\frac{2\pi 3t}{8}\Big)f_{k+1}^2\Big(\frac{3t}{8}\Big) + \sin^2\Big(\frac{2\pi3t}{8}\Big) f_{k+1}^2\Big(\frac{3t}{8} + \frac{3^{k+1}}{4}\Big) \Bigr]   \leq 1.\end{eqnarray*}

Thus, we have for all $t$, 
\begin{equation}\label{eqn:fsquared}  \cos^2\Big(\frac{2\pi 3t}{8}\Big)f_{k+1}^2\Big(\frac{3t}{8}\Big) + \sin^2\Big(\frac{2\pi3t}{8}\Big) f_{k+1}^2\Big(\frac{3t}{8} + \frac{3^{k+1}}{4}\Big) = 1 \end{equation}

Suppose for some value $\frac{3t}{8}$ we have $f_{k+1}^2(\frac{3t}{8}) < 1$.  Without loss of generality, we can also assume that $\cos(\frac{2\pi 3t}{8}) \neq 0$ since $f_{k+1}$ is a continuous function.  If $ f_{k+1}^2$ is strictly less than $1$ at some point, it is strictly less than one on an interval.  Therefore, we can choose a point from this interval where the cosine function is nonzero. 

At this point $t$, we have $\cos^2(\frac{2\pi 3t}{8})f_{k+1}^2(\frac{3t}{8}) < \cos^2(\frac{2\pi 3t}{8})$.  Since we also know that $f_{k+1}$ is bounded above by $1$ for all $t$, this results in a contradiction of Equation (\ref{eqn:fsquared}).  Therefore, we must have that $f_{k+1}(t) = 1$ for all $t$.  This proves that, given one spectral set $3^k\Gamma$ yielding an orthonormal basis of exponential functions, we must have all sets $3^k\Gamma$ yielding ONBs.
\end{proof}

\section{Acknowledgements}
The authors would like to thank Dorin Dutkay, Judith Packer, Noel Brady, and Erin Pearse for helpful 
conversations during the writing of this paper.  The clever proof for $p=2n-1$ in Remark \ref{rmk:examples}  was shown to us by 
Patrick Orchard, an undergraduate student at the University of Oklahoma.  
We are also grateful to the anonymous referees, whose thoughtful comments made this a better paper.


\bibliographystyle{amsalpha}
\bibliography{bernoulli}

\end{document}